\documentclass[12pt]{article}

\usepackage{amsmath,amsthm,amssymb}
\usepackage{hyperref}

\newtheorem{theorem}{\noindent Theorem}
\newtheorem{lemma}{\noindent Lemma}

\newtheorem{corollary}{\noindent Corollary}

\newtheorem{statement}{\noindent Proposition}

\newtheorem{definition-theorem}{\noindent Definition-Theorem}

\title{Cohomology in nonunitary representations of semisimple Lie groups (the group $U(2,2)$)}

\author{A.~M.~Vershik\thanks{St.~Petersburg Department of Steklov Institute of Mathematics and
St.~Petersburg State University, St.~Petersburg, Russia. Email: {\tt vershik@pdmi.ras.ru}. Supported by the RFBR grants
11-01-00677-a and 13-01-12422-ofi-m.}
\and
M.~I.~Graev\thanks{Institute for System Studies, Moscow, Russia. Email: {\tt graev\_36@mtu-net.ru}.
Supported by the RFBR grant 13-01-00190а.}}

\date{}

\begin{document}

\maketitle
\rightline{\textit{To the 100th birthday of our teacher Israel Moiseevich Gelfand}}
%\begin{flushright}{\it To the centenary of Israel Moiseevich Gelfand}\end{flushright}

\abstract{We suggest a method of constructing special nonunitary representations of semisimple Lie groups using representations of Iwasawa subgroups. As a typical example, we study the group
 $U(2,2)$.}

\section{Introduction: a survey of the theory of special representations}

\subsection{Groups of currents and special representations}

The group of currents
$G^X$, where $X$ is a topological space equipped with a Borel probability measure $m$ and $G$ is an arbitrary locally compact group, is the group of continuous maps  $X\rightarrow G$ with pointwise multiplication and with some integrability condition with respect to $m$. The study of representations of such groups is inspired both by representation theory itself and applications to mathematical physics. A well-known model of irreducible representations of current groups $G^X$  is the Fock, or Gaussian, model, in which a crucial role is played by nonzero first cohomology of the group of coefficients $(G$) with values in irreducible unitary representations of this group. Such an irreducible representation $T$ of the group $G$ in a Hilbert space $H$ has a nontrivial 1-cocycle, i.e., a continuous map $b:G\rightarrow H$ satisfying the condition
$$ b(g_1g_2)=T(g_1)b(g_2)+b(g_1) \quad \mbox{for any} \quad g_1, g_2 \in G $$
and the following nontriviality condition: there is no vector
$\xi\in H$ such that $b(g)=T(g)\xi-\xi$ for every $g\in G$.
Such representations are called {\it special}. The trivial representation is special for all groups having nontrivial complex characters, since an additive complex character is exactly a nontrivial 1-cocycle with values in the one-dimensional trivial complex representation.

In other terms, a special representation is a representation in which there exist
{\it almost invariant vectors}; the latter term means, for a representation of a group $G$ in a space $H$, that for every
$\epsilon >0$ and every compact subset $K$ in $G$, there exists a vector $h$ in $H$ such that
$\|U_k h - h\|<\epsilon$ for every $k \in K$, where $U_k$ is the operator in the representation corresponding to the element $k$.

By a theorem from \cite{VK}, every special representation of a compactly generated group is not Hausdorff separated from the trivial representation in the Fell topology. The Fell topology on the space of unitary representations of a locally compact group $G$ is defined as follows: an open neighborhood of a representation $\pi$ in a Hilbert space $H$ is determined by a number
$\epsilon$, a finite collection $h_1,\dots h_k$ of elements of $H$, and a compact subset $K$ in $G$; it consists of all unitary representations
$\rho$ of $G$ such that the space of $\rho$ contains elements
$f_1,\dots, f_k$ such that
$$\sup_{i=1,\dots, k; \;g \in K}\left\{|\langle\pi(g)h_i,h_i\rangle -\langle\rho(g)f_i,f_i\rangle|<\epsilon\right\}.$$
The converse is not true: a representation that is not separated from the trivial representation is not necessarily special, i.e., does not necessarily have non vanishing first cohomology.

In \cite{Sh}, Y.~Shalom proved the conjecture stated in
\cite{VK}: if for a locally compact group the trivial representation is not isolated in the space of all irreducible unitary representations (i.e., the group does not have Kazhdan's  property (T)), then it has at least one special representation. The proof of the main theorem in \cite{Sh} (see also \cite[Theorem 3.2.1]{Harp})  is not constructive, and hence does not provide a direct method of finding a special representation: we know very little about how to select special representations from the set of all representations that are not separated from (or ``glued'' to) the trivial representation.

We use the following terminology. The
{\it core} of a given representation $\pi$ is the set of all representations
$\rho$ such that arbitrary open neighborhoods of
$\pi$ and $\rho$ have a nonempty intersection. The core of the regular representation of an amenable group contains all irreducible representations, and this is a characteristic property of amenability. This notion is of special importance for irreducible representations and, in particular, for the trivial representation. Unless otherwise stated, by the core of a group we mean the core of its one-dimensional trivial representation.\footnote{A more detailed terminology is as follows. Representations lying in the core of the trivial representation are called {\it infinitesimal} (in a Leibniz-like sense, see
\cite{VK});  the {\it subcore}
of a given representation $\pi$ is the set of all representations
$\rho$ whose closure contains $\pi$ (correspondingly, the subcore of a group is the subcore of its trivial representation).
In other terms, this means that $\rho$ weakly contains $\pi$. The subcore is, obviously, a subset of the core. The interpretation of these notions in terms of {\it Hausdorff's separation axioms} is as follows: if the trivial representation lies in the core of $\rho$, then $T_{2}$ does not hold for these two elements; if $\rho$ lies in
the subcore, even $T_1$ does not hold;  $T_0$ always holds, since the trivial representation is closed.}

One says that  nonvanishing cohomology with values in a representation is reduced
 (see \cite{Sh}) if the corresponding cocycle is not a limit of trivial cocycles. This can be the case only for representations lying in the core, but not in the subcore.
For the semisimple groups $O(n,1), U(n,1)$, the non vanishing cohomology is reduced. The authors do not know whether there exist solvable groups satisfying this condition.
Finally, for completeness we mention another interesting notion --- that of groups with the Haagerup property (see \cite{Harp, Ha}): these are groups for which the trivial representation is not isolated (i.e., which do not have Kazhdan's property) but cocycles with values in a special representation satisfy a certain nondegeneracy  condition:  regarded as a map $$\beta:G\rightarrow H$$ from the group to a Hilbert space, the cocycle is proper, i.e., the preimages of bounded sets in $H$ are precompact in the group. This property holds  for all amenable groups, free groups (\cite{Ha}), the groups $O(n,1), U(n,1)$, and others (see \cite{Va}). Examples of groups that satisfy neither Kazhdan's property nor the Haagerup property are not yet sufficiently studied.

\subsection{Special representations of rank~$1$ groups and their solvable subgroups}
In the most important class of Lie groups --- that of semisimple groups --- only the groups $O(n,1)$ and $U(n,1)$, which are of rank~$1$, have special representations. The other semisimple groups (including, as proved by Kostant \cite{Ko}, even rank~$1$ groups $Sp(n,1)$) have Kazhdan's property~(T): their trivial representations are isolated, and the Fock model of constructing representations of the corresponding groups of currents does not apply. Note that Kostant's theorem on
$Sp(n,1)$ still has no geometric proof.

One may say that the analysis of representations of the groups of currents for
$O(n,1)$ and $U(n,1)$ is developed quite well. Studies in this direction began from the pioneer work by I.~M.~Gelfand and the authors of this paper, see
\cite{VGG1, VGG2, VGG3}. The general scheme of the Fock model, regardless the concrete group, was earlier considered by Araki
\cite{Ar} (see \cite{Par}), however, before the paper
 \cite{VGG1} there were no examples of semisimple groups for which the core is nontrivial.
 In these papers, as well as in \cite{Ism, Del, Gui, Ber}, irreducibility conditions for representations of current groups were found, and other properties of these groups were established. The key role was played by the ideas of  the paper \cite{VGG3}, in which a method was suggested, for semisimple groups of rank~$1$, of reducing the problem to a solvable subgroup, on the example of $SL(2,R)$. The elaboration of this idea by the authors of this paper during the last 10 years has led to new models of representations of current groups which are equivalent to the Fock one but are constructed from other (non-Gaussian) L\'evy measures (see \cite{VG1, VG2}). This has led to constructing the integral, and then Poisson and quasi-Poisson, models of representations of current groups (\cite{VG3, VG4}). There has also appeared the so-called infinite-dimensional Lebesgue measure \cite{Leb, V1, V2}, which is the most precise continual analog of the Lebesgue measure on a finite-dimensional positive octant. This measure is closely related to the gamma process and is of great interest in itself.

\subsection{Solvable subgroups of semisimple groups, and a refinement of the Iwasawa decomposition}
It is well known that for commutative and nilpotent groups the special representations are exhausted by the trivial representation (see \cite{Gui, Del}).  But even for solvable groups, this problem is not sufficiently studied. We are interested in a concrete class of solvable groups, which are subgroups of
$O(p,q),U(p,q)$, and our first example is the group  $U(2,2)$.
We are working with a solvable subgroup of $U(2,2)$ which should be called the {\it Iwasawa group}. Such a subgroup can be defined in an arbitrary connected semisimple real Lie group (see \cite{VO}, where it was called the ``maximal connected triangular subgroup''). The Iwasawa decomposition means that this subgroup is ``complementary'' to the maximal compact subgroup.

A solvable subgroup (like any amenable group) has a special representation, so that we obtain the following strategy of constructing a representation of the group of currents of a semisimple group $G$: first find a special representation, unitary or not, of this solvable subgroup and construct a representation of the group of its currents, and then try to extend the special representation to the whole semisimple group $G$ and construct an extension of the representation to the group of currents of $G$. As mentioned above, for groups of rank~$1$, this trick was first used in \cite{VGG3} in the case of
$SL(2,R)$, and then studied in detail in a recent series of papers of the authors (\cite{VG3, VG4}).  When passing from groups of rank~$1$ to higher ranks, this idea is still working.

Here we consider this problem on the concrete example of the group
 $U(2,2)$, keeping in mind the more general situation, which will be considered elsewhere. We describe in detail the Iwasawa subgroup (denoted by $P$ in what follows) for this case. The first question is, what are its special representations?
But here we encounter a new problem: for our plan to be viable, this special representation of the Iwasawa subgroup must be faithful.\footnote{A faithful (or nondegenerate) representation is a representation whose kernel is trivial.} The authors do not know for what groups this is the case. For example, nilpotent groups have no faithful special representations. This fact is of interest already for the Heisenberg group, and it is equivalent to a version of the uncertainty principle.\footnote{From discussions with V.~P.~Khavin, the first author inferred that this fact about the Heisenberg group apparently still has no purely analytical (rather than representation theoretic) proof.} For the affine group
 $\mbox{Aff}(\Bbb R)$, the infinite-dimensional representation (which is quasi-equivalent to the regular one) is a faithful special representation. But even for simplest three-dimensional solvable groups (in particular, for the group $S$, see below), the question is not trivial.

 \medskip
\noindent \textbf{Problem.}  What groups (in particular, what solvable groups) have a faithful irreducible unitary special representation? More precisely: when does such a representation exist for the Iwasawa subgroup of a semisimple real Lie group?

\medskip
Thus the main difficulty is to construct a faithful special representation, unitary or not, of the Iwasawa subgroup. {\it A construction of a nonunitary faithful special representation of the group $U(2,2)$ is the main result of this paper}, and the authors have no doubts that such a construction can be carried out for every real semisimple group. This makes it possible to extend the cocycle to the whole semisimple group and construct a representation of the group of currents, since, as follows from the results of the above-mentioned papers, there are models (for instance, the Poisson model, unlike the Fock one) of representations of groups of currents that do not  in any way rely on the unitarity of the original representation. The question of whether a cocycle can be constructed with values in a unitary representation of the Iwasawa subgroup is yet open.

In conclusion of this survey, we mention a somewhat different approach, which is closer to the original work on groups of currents and classical work on the representation theory of semisimple groups. Namely, it is well known that special representations of semisimple groups of rank~$1$  lie in the ``tail'' of the complementary series, and complementary series exist for every semisimple group. However, for groups of rank greater than~1,  unitary representations do not ``reach'' the trivial representation; more exactly, unitarity gets lost under deformations of the trivial representation. On the other hand, it is known that in some cases one can find an indefinite bilinear form invariant under the action of the group in such a nonunitary representation. The properties of the corresponding space with an indefinite metric are poorly studied, and the problem of the existence of a special representation, perhaps nonunitary, apparently has not been stated.
Our alternative strategy relies on the analytic continuation of unitary representations of the Iwasawa subgroups and, in particular, construction of a nonunitary special representation within this framework.
It is conceivable that these two approaches to the representation theory of the groups of currents for semisimple groups of rank greater than~1 may lead to different classes of representations.

\section{The group $U(2,2)$ and its Iwasawa subgroup}

\subsection{Description of the group $U(2,2)$}
As a first example, we consider the group
$U(2,2)$ of linear transformations of ${\Bbb C}^4$ preserving a fixed Hermitian form with signature $(2,2)$; here we choose the Hermitian form
 $$x_1\bar x_3+\bar x_1 x_3 +
 x_2\bar x_4+ \bar x_2 x_4.$$
The group $U(2,2)$ is one of the simplest examples of a semisimple Lie group whose real rank is greater than one (it is equal to 2). The groups of the form $U(p,q)$ are called pseudo-unitary. In this section, we refine the Iwasawa decomposition for this group and study the structure of the key object, the solvable Iwasawa subgroup $P$.

We will write elements of the group  $U(2,2)$ as $2\times 2$ block  matrices with $2\times 2$ blocks:
$$\left(
  \begin{array}{cc}
    g_{11} & g_{12}\\
    g_{21} & g_{22}\\
 \end{array}
\right),
$$
where $g_{ij}$ are complex $2\times 2$ matrices satisfying the relation
$$
g \sigma g^*= \sigma, \qquad \mbox{with} \quad \sigma =
\left(
  \begin{array}{cc}
    0 & e_2 \\
    e_2 & 0 \\
  \end{array}
\right).
$$
Here $e_2$ is the $2\times 2$ identity matrix and
$*$ stands for the conjugate transpose in the complex case, and for the transpose in the real case.

These relations are equivalent to the following relations between the blocks of the matrix $g$:
\begin{eqnarray*}
g_{12}g^*_{21}+g_{11}g^*_{22}&=&e_2,\\
g_{11}g^*_{12}+g_{12}g^*_{11}&=&0,\\
g_{22}g^*_{21}+g_{21}g^*_{22}&=&0.
\end{eqnarray*}

The real dimension of the group $U(2,2)$ is equal to 16. In what follows, the key role is played by the {\it solvable subgroup $P$ of $U(2,2)$ generated by the following two subgroups
$N$ and $S$}:

\begin{itemize}
\item the additive (commutative) group $N$ of skew-Hermitian block matrices of the form
$$
\left(
  \begin{array}{cc}
  e_2 & 0\\
    n & e_2 \\
  \end{array}
\right),
$$
where $n$ is a skew-Hermitian $2\times2$ matrix: $n+n^*=0$;

\item the solvable subgroup $S$ (of derived length 2) of block matrices of the form
$$
\left(
  \begin{array}{cc}
  s^{*{-1}} & 0\\
    0 & s \\
  \end{array}
\right),
$$
where
$s$ is a lower triangular complex matrix with positive diagonal entries:
$$
\left(
  \begin{array}{cc}
    r_1 & 0 \\
    r & r_2 \\
  \end{array}
\right),\qquad r_1, r_2>0, \;r \in \Bbb C.
$$
\end{itemize}

The real dimension of the groups $S$ and $N$ is equal to  4, and that of the group $P$ is equal to 8.

A general element of the group $P$ (a pair $(s,X)$) is as follows:
 $$\left(
     \begin{array}{cc}
      s^{*{-1}}&0 \\
       X & s \\
     \end{array}
   \right),
$$
where $X$ is a $2\times2$ matrix satisfying a condition that could be called the {\it relative skew-Hermiticity} (with respect to the matrix $s$):
$$sX^*+Xs^*=0.$$

The following assertion can be  checked directly.

\begin{statement} The group $U(2,2)$ is algebraically generated by the elements of the group $P$ and the involution $\sigma$; the intersection of the groups $N$ and $S$ consists of the identity element.
\end{statement}

The homogeneous space $U(2,2)/K$ where $K$ is the maximal compact subgroup of $U(2,2)$ is exactly the space of the subgroup $P$, which justifies calling it the Iwasawa subgroup. (This makes it possible to extend a cocycle from $P$ to the whole group $U(2,2)$, see below.)

Since the group $S$ acts in an obvious way on the additive group $N$ of skew-Hermitian matrices according to the rule
  $$ n\mapsto sns^*,\qquad  s\in S, \;n\in N,$$
we can define the semidirect product
$Q=S\rightthreetimes N$ of $S$ and $N$. One can directly check the following important assertion.

\begin{theorem}
The groups $P$ and $Q$ are canonically isomorphic. The isomorphism
 $I: P \rightarrow Q$ is given by the formula
$$I:(s,X)\rightarrow (s, Xs^*)$$
(the left-hand side is an element of the subgroup
$P \subset U(2,2)$, and the right-hand side is an element of the semidirect product
 $Q$, the matrix  $Xs^*$ being obviously skew-Hermitian). The inverse isomorphism is as follows: $(s,n)\rightarrow (s,{ns^*}^{-1}),$ where $n$ is skew-Hermitian.
\end{theorem}

\begin{proof}
We check that the multiplication is homomorphic:
$$I((s_1, X_1)\circ (s_2, X_2)) = I(s_1 s_2, X_1 {s_2^*}^{-1} + s_1X_2)=$$
$$=(s_1 s_2,X_1s_1^*+s_1X_2s_2^*s_1^*)=I(s_1,X_1)I(s_2,X_2).$$
 \end{proof}

The map $I$ sends $P$ to the group whose matrix realization consists of the collection of pairs
 $(s,X)$ satisfying the above condition. Hence the isomorphism of the semidirect product $Q$ of $S$ and $N$ with its matrix realization, i.e., the group $P$, does not coincide with the direct product of the identity isomorphism of the subgroup $S$  and the identity isomorphism of the normal subgroup $N$. Nevertheless, the group $P$ is the semidirect product $S\rightthreetimes N$.\footnote{In the case of a group of rank~$1$, the matrix realization of the Iwasawa subgroup is somewhat simpler; for example, for the group $SL(2)$ the Iwasawa subgroup  $P$ is the subgroup of triangular matrices
$$\left(
                              \begin{array}{cc}
                                s^{-1} & o \\
                                  n    & s \\
                              \end{array}
                            \right),\qquad s\in {\Bbb R}_+, \;n\in \Bbb R,$$
and the structure of the semidirect product agrees with the ordinary matrix representation, since the group regarded as a set is the direct product of the normal subgroup
$$\left(
                                                                               \begin{array}{cc}
                                                                                 1 & 0 \\
                                                                                 n & 1 \\
                                                                               \end{array}
                                                                             \right)$$
and the subgroup $$\left(
              \begin{array}{cc}
                s^{-1} & 0 \\
                0      & s \\
              \end{array}
            \right).$$}

A general element of the commutant of the group $S$ in the above notation is as follows:
$$
 \left(
  \begin{array}{cc}
   1 & 0 \\
   r & 1 \\
  \end{array}
  \right),
$$
where $r$ is a complex number. The derived length of the solvable group $S$ is equal to 2. Thus the commutant of the group $P$ is the semidirect product of
$\Bbb C$ and the commutative group of skew-Hermitian matrices (with a nontrivial action of $\Bbb C$).

\begin{corollary}
The derived length of the solvable group $P$ is equal to~$3$.
\end{corollary}

\subsection{Representations of the Iwasawa subgroup and almost invariant measures}

Our aim is to study the special representation of the subgroup $P$ and then extend it to the whole group $U(2,2)$. We will study representations of the group $P$ regarded as a semidirect product.

The group $N$ is isomorphic to the group $\hat N$ of all its continuous characters, and we have the conjugate action of the group $S$ on $\hat N$;
moreover, the direct and conjugate actions on the group $N$ (identified with its dual group) coincide:
$$
n\rightarrow sns^*; \qquad\chi\rightarrow \chi_s \quad \mbox{where}\quad   \chi_s(n)=\chi(sns^*).
$$

\begin{statement}
The action of the group $S$ by automorphisms on the group $\hat N$ (and on the group $N$) has four orbits of positive measure; they are parametrized by the signs of the imaginary parts of the diagonal entries, i.e., are the orbits of the elements
$$\left(
                                                                     \begin{array}{cc}
                                                                       \pm i & 0 \\
                                                                        0    & \pm i \\
                                                                     \end{array}
                                                                  \right),$$
 respectively.
On every orbit, the action of $S$ is free and faithful, so that each orbit can be identified with the group $S$; then the action coincides with the action of $S$ on itself by right translations.
\end{statement}

Note that there are also orbits of smaller dimension, which have zero measure, but we will not need them.

It is clear that all four actions of the group $S$ on the nondegenerate orbits are topologically isomorphic and the corresponding representations of the semidirect product differ only by a ${\Bbb Z}_2$-valued cocycle which acts as a multiplication,
so that it suffices to consider only one (any) orbit.

A unitary representation of the semidirect product of a group and a commutative group (in our case,   $P=S\rightthreetimes N$) has the following canonical realization. Consider a probability measure $\mu$ on the group of characters (i.e., on
 $\hat N$) that is quasi-invariant with respect to the action (of the group $S$). All such measures are equivalent, since $S$ is locally compact and its action is transitive. Hence in the Hilbert space $L^2_{\mu}(\hat N)$ we can define the unitary representation of the group~$P$ induced by the above action of $S$ and the representation of $N$ in which an element of $N$ acts as the multiplication by the corresponding character or $\hat N$.
 General representations are realized in the vector-valued space $L^2_{\mu}(\hat N)$, but we do not consider them.

The irreducibility of the above representation of the semidirect product is equivalent to the ergodicity of the measure $\mu$, and, by the above, we could assume that the orbit is the group $S$ itself, i.e., consider the representation of  $S$ in the space
 $L^2_{\mu}(S)$ over a measure
$\mu$ quasi-invariant with respect to the right action of the group.

These representations belong to the core of the group $P$; this follows from the fact that each of them is quasi-equivalent to the regular representation of the group, which (by the amenability of $S$) weakly contains the trivial representation. Besides, the core contains the trivial representation, as well as the special representations of the group $S$, which can be regarded as representations of $P$, since $S=P/N$. It is not known whether the core is exhausted by these representations, nor whether the four constructed representations are special for the group $P$, or, in other words, whether they have an almost invariant vector. But first we find out the structure of the special representations of the group $S$.

\begin{lemma}
The group $S$ has a continuum of unitary representations parametrized by the points
of $\Bbb C$ lying on the unit circle (characters). It has no faithful special representations.
\end{lemma}

\begin{proof}
The affine group $\mbox{Aff}(\Bbb R)$, i.e., the group of matrices
 $$\left(
                                                    \begin{array}{cc}
                                                      e^a & 0 \\
                                                      b & e^{-a} \\
                                                    \end{array}
                                                  \right), \qquad a, b \in {\Bbb R},$$
is the Iwasawa subgroup for
$SL(2,R)$ and plays the same role as the group $P$ does for
 $U(2,2)$. Note that, regarded as a subgroup of
$SL(2,R)$, it is the semidirect product
of ${\Bbb R}_+$ and $\Bbb R$ that agrees with the matrix representation; as mentioned above, in our situation this is not the case. It has (two) faithful special unitary representations, and they can be extended to a special representation of
$SL(2,R)$.

Consider the group $S$:
                     $$   \left(
                     \begin{array}{cc}
                       r_1 & 0 \\
                       r & r_1 \\
                    \end{array}
                  \right),\qquad
                  r\in {\Bbb C}, \quad r_1,r_2>0;
                  $$
note that if we fix a value of the determinant  ($=r_1\cdot r_2$)  and  a unitary character on the group   ${\Bbb C}={\Bbb R}^2$,
the group $S$ can be mapped isomorphically to the group $\mbox{Aff}(\Bbb R)$, and hence all special representations of  $\mbox{Aff}(\Bbb R)$ can be lifted to representations of $S$; all of them are not faithful; the group~$S$ has no other special representations.
\end{proof}

However, we are interested not as much in the group $S$, but in the group $P$. Conjecturally, special unitary representations of $P$ can be constructed using special (nonfaithful) representations of $S$ together with the representations of $P$ constructed above. At the moment, the question  of whether such representations exist is open. For completeness, we describe a model of the Hilbert space of a special unitary representation for the group
 $SL(2,R)$ and its triangular subgroup ($P$). Consider the space
 $L^2_m({\Bbb R}_+)$  (where $m$ is the Lebesgue measure on the half-line). It is more convenient to pass to the Fourier transform, and then the required representation of the triangular subgroup (written as the group of transformations
 $x \mapsto e^{\beta}x+a$ with $\beta, a \in \Bbb R$)
can be realized in $L^2_{\hat m}(\Bbb R)$ (where ${\hat m}$  is the Lebesgue measure on the line) as follows:
$$(U_{a,\beta}F)(z)=\exp\{iae^z\} F(z+\beta), \quad z \in \Bbb R. $$
An almost invariant vector in this model is an arbitrary function $f$ satisfying, for any
 $t, a,b \in \Bbb R$, the conditions
$$ f(x)=0 \mbox{ if } x>t \in {\Bbb R};$$
$$f \notin L^2; \quad (1-\exp\{ie^zb\})f \in L^2; \quad [f(\cdot)-f(\cdot+a)] \in L^2.$$
Another, more popular, description of the special representation (see
\cite{VGG1}) in a space of analytic functions is related to the limit of complementary series representations as they tend to the trivial representation.

\subsection{Almost invariant measures and nonunitary representations}

We say that a measure $\nu$ on a group $S$  is {\it (right) almost invariant} if it is infinite, absolutely continuous with respect to the right Haar measure on $S$ (and hence quasi-invariant with respect to the right translations
 $s \mapsto ss_0$), and its derivatives $\frac{d\nu(ss_0)}{d\nu(s)}$
are defined and bounded for every
 $s_0\in S$. (By the above isomorphism
$S\rightarrow H$, where $H$ is an arbitrary nondegenerate $S$-orbit on the group of characters $\hat N$, this definition can be translated to measures concentrated on any of the nondegenerate orbits of the group $\hat N$).

The almost invariance condition is obviously satisfied for the Lebesgue measure on $S$, i.e.,
$$ds=ds_{11} ds_{22} ds_{21} d\bar s_{21},$$
since $d(ss_0)=\pi(s)ds$, where $\pi(s)=s_{11}^3s_{22}$. It follows that this condition holds for any measure of the form
$d\nu(s)=u(s)ds$, where $u(s)$ is an arbitrary function such that the ratio
$\frac {u(ss_0)}{u(s)}$ is a bounded function for every
$s_0\in S$. In particular, it holds for the measure
  $\mu$ that is invariant under the right translations (the Haar measure):
$$
d\mu(s)=\pi^{-1}(s)ds, \quad \pi(s)=s^3_{11}s_{22}.
$$
However,  for our purposes it is convenient to consider another measure.

Assume that a group $G$ acts on a space $X$, and we are given two equivalent
$G$-quasi-invariant measures
$\mu$ and $\nu$ on $X$. Assume that the density of one measure with respect to the other one is bounded away from zero and infinity. In the spaces $L^2_{\mu}(X)$ and $L^2_{\nu}(X)$ we consider the representation of the group $G$   by the substitutions $(U_g f)(\cdot)=f(g\cdot)$ and the natural representation of the group of multiplicators with absolute value equal to one.
The well-known isometry between these spaces, which multiplies a function by the square root of the density of one measure with respect to the other one, commutes with the multiplicators, but, in general, does not commute with the action of the group. This isometry is widely used to correct the action; for instance, if one of the measures is invariant, and thus determines a unitary representation of the cross product, then the corrected action also becomes unitary.

Let $\nu$ be an almost invariant measure on $H$; consider a nonunitary representation of the group $P$ in the Hilbert space $L^2(S,\nu)$, i.e., the space of functions $F$ on $S$ with the norm $\|F\|^2=\int_S |F(s)|^2 d\nu(s)<  \infty$. By definition, the operators corresponding to the elements of the subgroups $N$ and $S$ are given by the following formulas:
\begin{eqnarray}
(T(n)F)(s)&=&\chi_k(n,s)F(s) \quad\mbox{ for } n\in N; \\
(T(s_0)F)(s)&=&F(ss_0) \quad\mbox{ for } s_0\in S.
\end{eqnarray}
Here $\chi_k(n,s)$ is the image of a character
$\chi(\cdot) \in \hat N$ regarded as a function on $S$ under the (unique) isomorphism between the orbit of $S$ in $\hat N$ indexed by $k=1,2,3,4$ and $S$ preserving the action of $S$; the difference between the four orbits reduces to multiplying the image by
$\pm i$ in each of the variables. It follows from the definition that the operators
 $T(n)$ are unitary and the operators
$T(s_0)$ are bounded, by the almost invariance of the measure $\nu$. One can rewrite the formulas in a more compact form:
\begin{eqnarray}
(T(n)f)(s)&=&\chi(sns^*)f(s) \quad\mbox{ for } n\in N; \\
(T(s_0)f)(s)&=&f(ss_0) \quad\mbox{ for } s_0\in S,
\end{eqnarray}
where $\chi$ is a fixed character of the $S$-orbit $H$ on the group of characters
 $\hat N$.

It is not difficult to verify that the operators of the subgroups $N$ and $S$ together generate a representation of the whole group $P$ in the space
$L^2(S,\nu)$. In particular, if $\nu=\mu$ is the Haar measure on $S$, this representation is unitary. Denote these representations of the group $P$ by
 $\pi_k$, $k=1,2,3,4$; since all four representations essentially differ from one another only by a choice of a character on the orbits, we omit the index $k$.

\begin{theorem}
The nonunitary representations $\pi$ of the group $P$ defined above are operator irreducible and space irreducible.
The representations corresponding to different measures and different indices $k$ are space (unitarily) equivalent if and only if the measures $\nu$ differ by a factor $(\nu'=c\nu)$ and the indices $k$ coincide.
\end{theorem}

An important question for us is how the cohomology depends on the measure. We emphasize that changing the measure and, in particular, the unitarization of representations (see above) does not induce an isomorphism of the cohomology groups
$H^1(G,\pi_{\mu})$ and  $H^1(G,\pi_{\nu})$ where $\pi_{\mu}$,  $\pi_{\nu}$ are the representations corresponding to the measures
$\mu$ and $\nu$, since a space isometry does not  in general send a cocycle of a group with values in one space to a cocycle with values in another space. In other words, for a given action of the group, the cohomologies with values in the Hilbert space
$L^2_{\mu}(X)$ for various almost invariant measures $\mu$ are in general different. In the next section we choose an almost invariant measure for which the cohomology is nontrivial.

\subsection{A faithful nonunitary special irreducible representation of the Iwasawa group}
%\subhead{2.4. A faithful nonunitary special irreducible representation of the Iwasawa group}
An important question for us is how the cohomology with values in $L_\nu^2(X)$
depends on the almost invariant measure $\nu$.
We emphasize that changing the replacement of the measure by an equivalent one
and, in particular, the unitarization of representations (see above)
does not induce an isomorphism of the cohomology groups
$H^1(G,\pi_{\mu})$ and $H^1(G,\pi_{\nu})$, where $\pi_{\mu}$, and $\pi_{\nu}$
are the representations corresponding to the measures
$\mu$ and $\nu$, since a space isometry does not in generally send
a cocycle of a group with values in one space to a cocycle with values in another
the other space. In other words, for a given action of the group,
the cohomologies with values in the Hilbert space
$L^2_{\mu}(X)$ for various almost invariant measures $\mu$ are in generally different.
This explains our choice of an almost invariant measure in what follows.
In the next section we choose an almost invariant measure for which the cohomology is nontrivial.

Let us fix an almost invariant measure $\nu$ on the space $S$ and introduce the space $Z_{\nu}$
of measurable functions on $S$ satisfying the two conditions
\begin{align*}
\int_S|f(ss_0)-f(s)|^2\,d\nu(s)&<\infty\quad\text{for any}\,s_0\in S,\\
\int_S|(\chi(sns^*)-1)f(s)|^2\,d\nu(s)&<\infty\quad\text{for any}\,n\in N.
\end{align*}

Obviously, $L^2_{\nu}(S) \subset Z_{\nu}$.

We treat the elements of $f \in Z_{\nu}$ as coboundaries and the space of functions  of the form $b(g)=T(g)f-f$ as
the space of cocycles. If $f\in L^2_{\nu}(S)$, then $b(g)=T(g)f-f$ is  a cocycle cohomologous to zero.
This implies the following assertion.

\begin{lemma}
A representation of the group $P$ in the space $L^2_{\nu}(S)$ is special if and only if
$$
L^2_{\nu}(S)\varsubsetneq Z_{\nu},
$$
i.e., if there exists a coboundary not lying in $L^2_{\nu}(S)$. Thus, the first cohomology group has the form
$$
H^1(P;L^2_{\nu}(S) =Z_{\nu}/Z_{\nu}L^2_{\nu}(S).
$$
\end{lemma}

A measure $\nu$ is said to be special if $H^1(P;L^2_{\nu}(S))\neq 0$, i.e., if there exist coboundaries not
lying in  $L^2_{\nu}(S)$. The authors do not know whether
the Haar measure $m$ on $S$ is special and, thereby, whether
the natural unitary representation of the group $P$ on $L^2_{m}(S)$ is special.
For this reason, we use a different almost invariant measure to
construct a special, but not unitary, representation of $P$.

In what follows, we fix an almost invariant measure $\nu$ of the following form:
$$
d\nu(s)=|s|^{-4}\,ds,\quad\text{where }\,|s|^2=tr(s^*s)=s_{11}^2+s_{22}^2+|s_{21}|^2.
$$

It is convenient to write this measure in polar coordinates on $S$.
To do this end, note that the variety of elements $\omega \in S$
with norm $|\omega|=1$ is equivalent to a domain on the unit sphere in $\mathbb{R}^4$.
We define spherical coordinates of a matrix $s\in S$ we mean in equation 3 as
the number $r=|s|$ and the matrix $\omega=|s|^{-1}s$. Then $s=r\omega$,
and the expression for $\nu$ in polar coordinates has the form
$$
d\nu(s)=r^{-1}\,dr\,d\omega,
$$
where $d\omega$ is the invariant measure on the sphere.

The following assertion can be checked verified directly.

\begin{theorem} The representation $\pi$ of the group $P$ in the Hilbert space
$L^2(S,\nu)$, where $d\nu(s)=|s|^{-4}ds$, is special and has a nontrivial
cocycle of the form
\begin{equation}
b(g)=T(g)f-f,\quad\text{where }\,f(s)=e^{-|s|/2}.\tag{4}
\end{equation}
\end{theorem}

\subsection{Extending the special nonunitary representation of the subgroup $P$ to the whole group $U(2,2)$}

It remains to check that the special representation can be extended to a nonunitary representation of the whole group
$U(2,2)$. The construction of the required extension is based on the following property of the group $U(2,2)$. Every element $g\in U(2,2)$ can be uniquely written as a product
$g=pk$, $p\in P$,  $k\in K$, where $K$ is the maximal compact subgroup, which consists of the elements $k\in U(2,2)$ satisfying the relation
$kk^*=e$, i.e., the subgroup of block matrices of the form
$$k=\left(
      \begin{array}{cc}
       \alpha  & \beta \\
        \beta &\alpha \\
      \end{array}
    \right)
$$
where $\alpha\alpha^*+\beta\beta^*=e_p$ and $\alpha\beta^*+\beta\alpha^*=0$ (the Iwasawa decomposition).

Let $T$ be the special representation of the subgroup $P$ in the Hilbert space
$H=L^2(S,\nu)$ defined in Theorem~2, and $b$ be the nontrivial cocycle defined by~(8). Denote by $H_0$ the linear invariant subspace in $H$ spanned by the vectors
$b(p)$, $p\in P$.

\begin{lemma}
 {\rm1.} The subspace  $H_0$ is dense in $H$.

{\rm 2.} The vectors $b(p)$, $p\ne e$, are linearly independent (the nondegeneracy property).

Let $T(k)$, $k\in K$, be the operators defined on the set of vectors $b(p)$, $p\in P$, by the formula
$$T(k)b(p)=b(p'),$$
where $p'\in P$ is defined by the relation $kp=p'k'$, $k'\in K$.

{\rm 3.} The operators  $T(k)$ satisfy the group relation
$T(k_1k_2)b(p)=T(k_1)T(k_2)b(p)$ for any $k_1,k_2\in K$ and $p\in P$,
and hence generate a representation of the subgroup $K$ in the subspace
 $H_0$.
\end{lemma}

\begin{theorem}
The operators $T(k)$, $k\in K$, together with the operators
$T(p)$, $p\in P$, generate a representation of the whole group
  $U(2,2)$ in the subspace $H_0$, in which the operator corresponding to the involution $\sigma=\left(
                                                              \begin{array}{cc}
                                                                0 & e_p \\
                                                                e_p& 0 \\
                                                              \end{array}
                                                            \right)$
is defined on the set of vectors
 $b(p)$ by the following formula:
\begin{equation}
T(\sigma)b(p)=b(\hat p), \qquad\hat p\in P,
\end{equation}
where $\hat p$ is uniquely determined by the relation
$\hat p \hat p^*= \sigma pp^*\sigma.$ In this extension, the operators corresponding to the elements of the subgroup $P$ are unitary, the operators corresponding to the elements of the subgroup $K$ are bounded, and the operator corresponding to the involution is unbounded and cannot be extended to the whole space $H$. The extension of the $1$-cocycle $b$ from the group $P$ to the group $U(2,2)$ is given by the formula
\begin{equation}
b(pq)=b(p) \mbox{ for any } p\in P \mbox{ and } q\in Q.
\end{equation}
\end{theorem}

Applications of the described constructions to representations of the group of currents will be described elsewhere. In this paper we restrict ourselves to the group  $U(2,2)$ just for methodological reasons: our aim is to give a simple example of the general theory, which, in the authors' opinion, covers a wide class of semisimple Lie groups and the corresponding groups of currents.

Translated by N.Tsilevich.
\end{document}